\documentclass[12pt]{amsart}
\usepackage{amsmath,amsopn,amsfonts,amsthm,amssymb,xcolor}
\usepackage[colorlinks=true, pdfstartview=FitV, linkcolor=blue, citecolor=blue, urlcolor=blue]{hyperref}
\usepackage{mathtools}
\usepackage{setspace}
\usepackage{fullpage}

\usepackage{tikz}
\usepackage{tikz-cd}
\usetikzlibrary{arrows,calc} 

\tikzstyle{dot} = [inner sep=0.5mm,circle,draw,minimum size=1mm]
\tikzstyle{continued}=[dashed]

\newtheorem{thm}{Theorem}[section]
\newtheorem*{thm*}{Theorem}
\newtheorem*{lemma*}{Lemma}
\newtheorem*{prop*}{Proposition}
\newtheorem{lemma}[thm]{Lemma}

\newtheorem{prop}[thm]{Proposition}

\theoremstyle{definition}

\theoremstyle{remark}
\newtheorem*{rem*}{Remark}
\newtheorem*{ex*}{Example}

\usepackage{manfnt}
\newenvironment{ex}{\refstepcounter{thm}\begin{proof}[Example \emph{\thethm}]}{\end{proof}}
\newenvironment{rem}{\refstepcounter{thm}\begin{proof}[Remark \emph{\thethm}]}{\end{proof}}

\numberwithin{equation}{section}

\title{A short proof of the finiteness of Dynkin friezes}
\author{Greg Muller}
 
\begin{document}

\begin{abstract}
We give a short and elementary proof that every Dynkin diagram admits finitely many (positive integral) friezes. This was originally proven in \cite{GM22} using the geometry of cluster algebras. The proof here provides an explicit (albeit inefficient) bound on values.

\end{abstract}

\maketitle

\section{Recollections}

Let $\Gamma$ be a (finite-type) Dynkin diagram 
with vertices indexed from $1$ to $n$, and let $\mathsf{C}$ be the Cartan matrix of $\Gamma$. 
A \textbf{(positive integral) frieze of type $\Gamma$} is a function\footnote{Here and throughout, $\mathbb{Z}_{[1,n]}:=\{1,2,...,n\}$ and $\mathbb{Z}_{\geq1}:=\{1,2,3,...\}$.}
\[ F: \mathbb{Z}_{[1,n]} \times \mathbb{Z} \rightarrow  \mathbb{Z}_{\geq1}\]
such that the \textbf{mesh relations} hold; that, is, for each $(i,k)\in \mathbb{Z}_{[1,n]} \times \mathbb{Z}$, 
\[ F_{i,k}F_{i,k+1} = 1 + \prod_{j<i} F_{j,k}^{-\mathsf{C}_{i,j}}\prod_{j>i} F_{j,k+1}^{-\mathsf{C}_{i,j}} \]

A frieze can be visualized as values on the vertices of the associated \textbf{repetition quiver}. This is the infinite quiver with vertex set $\mathbb{Z}_{[1,n]}\times \mathbb{Z}$ and arrows $(i,k)\rightarrow (j,k)\rightarrow (i,k+1)$ for all $k\in \mathbb{Z}$ whenever $i<j$ and $-C_{i,j}<0$.
For simply-laced $\Gamma$, the mesh relations are equivalent to requiring that adjacent values in the same row have product equal to 1 plus the product of the intermediate values (see Figure \ref{fig: E8frieze}).

\begin{figure}[h!t]
 \begin{tikzpicture}[scale=.6,x=0.025cm,y=-0.025cm]
\begin{scope}[every node/.style={inner sep=4pt,font=\footnotesize}]
    \node (2_a) at (-30,339) {$\dots$};
    \node (4_a) at (-30,439) {$\dots$};
    \node (6_a) at (-30,469) {$\dots$};
    \node (7_a) at (-30,539) {$\dots$};

    \node (1_0) at (20,289) {$2$};
    \node (2_0) at (70,339) {$5$};
    \node (3_0) at (20,389) {$7$};
    \node (4_0) at (70,439) {$18$};
    \node (5_0) at (20,489) {$41$};
    \node (6_0) at (70,469) {$6$};
    \node (7_0) at (70,539) {$11$};
    \node (8_0) at (20,589) {$4$};
    \node (1_1) at (120,289) {$3$};
    \node (2_1) at (170,339) {$8$};
    \node (3_1) at (120,389) {$13$};
    \node (4_1) at (170,439) {$21$};
    \node (5_1) at (120,489) {$29$};
    \node (6_1) at (170,469) {$5$};
    \node (7_1) at (170,539) {$8$};
    \node (8_1) at (120,589) {$3$};
    \node (1_2) at (220,289) {$3$};
    \node (2_2) at (270,339) {$5$};
    \node (3_2) at (220,389) {$13$};
    \node (4_2) at (270,439) {$18$};
    \node (5_2) at (220,489) {$29$};
    \node (6_2) at (270,469) {$6$};
    \node (7_2) at (270,539) {$11$};
    \node (8_2) at (220,589) {$3$};
    \node (1_3) at (320,289) {$2$};
    \node (2_3) at (370,339) {$3$};
    \node (3_3) at (320,389) {$7$};
    \node (4_3) at (370,439) {$16$};
    \node (5_3) at (320,489) {$41$};
    \node (6_3) at (370,469) {$7$};
    \node (7_3) at (370,539) {$15$};
    \node (8_3) at (320,589) {$4$};
    \node (1_4) at (420,289) {$2$};
    \node (2_4) at (470,339) {$5$};
    \node (3_4) at (420,389) {$7$};
    \node (4_4) at (470,439) {$18$};
    \node (5_4) at (420,489) {$41$};
    \node (6_4) at (470,469) {$6$};
    \node (7_4) at (470,539) {$11$};
    \node (8_4) at (420,589) {$4$};
    
    \node (1_5) at (520,289) {$\dots$};
    \node (3_5) at (520,389) {$\dots$};
    \node (5_5) at (520,489) {$\dots$};
    \node (8_5) at (520,589) {$\dots$};
\end{scope}
\begin{scope}[every node/.style={fill=white,font=\scriptsize},every path/.style={-{Latex[length=1.5mm,width=1mm]}}]
    \path (2_a) edge[] (1_0);
    \path (2_a) edge[] (3_0);
    \path (4_a) edge[] (3_0);
    \path (4_a) edge[] (5_0);
    \path (6_a) edge[] (5_0);
    \path (7_a) edge[] (5_0);
    \path (7_a) edge[] (8_0);
    \path (1_0) edge (2_0);
    \path (3_0) edge (2_0);
    \path (3_0) edge (4_0);
    \path (5_0) edge (4_0);
    \path (5_0) edge (6_0);
    \path (5_0) edge (7_0);
    \path (8_0) edge (7_0);
    \path (2_0) edge[] (1_1);
    \path (2_0) edge[] (3_1);
    \path (4_0) edge[] (3_1);
    \path (4_0) edge[] (5_1);
    \path (6_0) edge[] (5_1);
    \path (7_0) edge[] (5_1);
    \path (7_0) edge[] (8_1);
    \path (1_1) edge (2_1);
    \path (3_1) edge (2_1);
    \path (3_1) edge (4_1);
    \path (5_1) edge (4_1);
    \path (5_1) edge (6_1);
    \path (5_1) edge (7_1);
    \path (8_1) edge (7_1);
    \path (2_1) edge[] (1_2);
    \path (2_1) edge[] (3_2);
    \path (4_1) edge[] (3_2);
    \path (4_1) edge[] (5_2);
    \path (6_1) edge[] (5_2);
    \path (7_1) edge[] (5_2);
    \path (7_1) edge[] (8_2);   
    \path (1_2) edge (2_2);
    \path (3_2) edge (2_2);
    \path (3_2) edge (4_2);
    \path (5_2) edge (4_2);
    \path (5_2) edge (6_2);
    \path (5_2) edge (7_2);
    \path (8_2) edge (7_2);
    \path (2_2) edge[] (1_3);
    \path (2_2) edge[] (3_3);
    \path (4_2) edge[] (3_3);
    \path (4_2) edge[] (5_3);
    \path (6_2) edge[] (5_3);
    \path (7_2) edge[] (5_3);
    \path (7_2) edge[] (8_3);
    \path (1_3) edge (2_3);
    \path (3_3) edge (2_3);
    \path (3_3) edge (4_3);
    \path (5_3) edge (4_3);
    \path (5_3) edge (6_3);
    \path (5_3) edge (7_3);
    \path (8_3) edge (7_3);
    \path (2_3) edge[] (1_4);
    \path (2_3) edge[] (3_4);
    \path (4_3) edge[] (3_4);
    \path (4_3) edge[] (5_4);
    \path (6_3) edge[] (5_4);
    \path (7_3) edge[] (5_4);
    \path (7_3) edge[] (8_4);
    \path (1_4) edge (2_4);
    \path (3_4) edge (2_4);
    \path (3_4) edge (4_4);
    \path (5_4) edge (4_4);
    \path (5_4) edge (6_4);
    \path (5_4) edge (7_4);
    \path (8_4) edge (7_4);
    \path (2_4) edge[] (1_5);
    \path (2_4) edge[] (3_5);
    \path (4_4) edge[] (3_5);
    \path (4_4) edge[] (5_5);
    \path (6_4) edge[] (5_5);
    \path (7_4) edge[] (5_5);
    \path (7_4) edge[] (8_5);
\end{scope}
 \end{tikzpicture}
\caption{A positive integral frieze of type $E_8$}
\label{fig: E8frieze}
\end{figure}
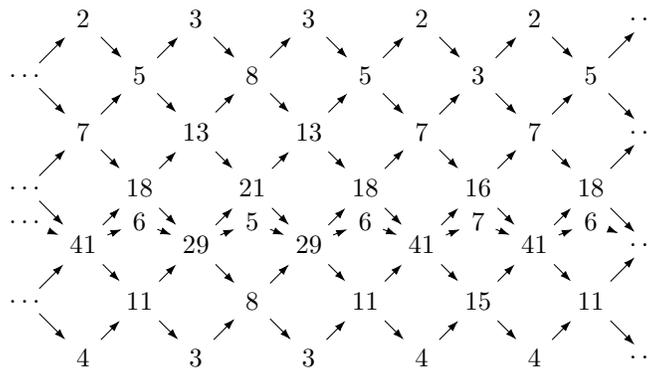

\begin{rem}
This construction implicitly orients the edges in $\Gamma$ via the ordering of the indices to avoid a choice of orientation used in other definitions (which is ultimately irrelevant).
%
\end{rem}


A thoughtful discussion of the motivations and history of Dynkin friezes can be found in \cite{MG15}, so we will skip this. We will need the following fundamental result.


\begin{thm}[{\cite[Theorem 2.7]{MG15}}]
For each Dynkin diagram $\Gamma$, there is an integer $p$ such that every frieze of type $\Gamma$ is $p$-periodic; that is, $F_{i,k}=F_{i,k+p}$.
\end{thm}


Friezes of type $A_n$ were introduced by Coxeter and Conway in \cite{CC73a}, who showed they were in bijection with triangulations of an $(n+3)$-gon, and are therefore counted by the Catalan numbers. This suggests that counts of friezes in other Dynkin types should be meaningful; however, the finiteness of such friezes was open until recently.


\begin{thm}\cite{GM22}
For each Dynkin diagram $\Gamma$, there are finitely many friezes.
\end{thm}

In \cite{GM22}, finiteness was proven non-constructively using the geometry of cluster algebras; specifically, by showing they form a closed and discrete subset of a compact set (the \emph{superunitary region} of the associated cluster algebra). In this note, we give an elementary proof of finiteness by explicitly bounding the size of the entries of a frieze.



\section{Finiteness}

\subsection{Average logarithms}

Consider a positive integral frieze $F$ of type $\Gamma$ with period $p$. For each vertex $i$ of $\Gamma$, let $a_i(F)$ denote the average of the base 2 logarithms of the entries in the $i$th row of $F$; that is,
\[ a_i(F) := 
\frac{1}{p} \sum_{k=1}^p 
\log_2(F_{i,k}) 
=
\frac{1}{p} 
\log_2\left( \prod_{k=1}^p F_{i,k}\right) 
\]
By periodicity, the sum could equivalently be taken over any sequence of $p$-many consecutive entries in the $i$th row of $F$. Let $\vec{\mathbf{a}}(F)$ denote the $n$-vector whose $i$th coordinate is $a_i(F)$.




\begin{lemma}\label{lemma}
Each entry of the product  $\mathsf{C}\vec{\mathbf{a}}(F)$ lies in the half-open interval $(0,1]$.
\end{lemma}


\begin{proof}
For each $i\in \Gamma_0$, $2a_i(F)$ is equal to 
\[
\begin{aligned}
\frac{1}{p} \log_2\left( \prod_{k=1}^p F_{i,k}^2\right) 
=\frac{1}{p} \log_2\left( \prod_{k=1}^pF_{i,k}F_{i,k+1}\right) 
=\frac{1}{p} \log_2\left( \prod_{k=1}^p \left(1+ 
\prod_{j<i} F_{j,k}^{-\mathsf{C}_{i,j}}\prod_{j>i} F_{j,k+1}^{-\mathsf{C}_{i,j}}
\right)\right) 
\end{aligned}
\]
Since each $F_{i,k}\geq1$, we have bounds
\[ 
\prod_{j<i} F_{j,k}^{-\mathsf{C}_{i,j}}\prod_{j>i} F_{j,k+1}^{-\mathsf{C}_{i,j}}
<
1+ 
\prod_{j<i} F_{j,k}^{-\mathsf{C}_{i,j}}\prod_{j>i} F_{j,k+1}^{-\mathsf{C}_{i,j}}
\leq 2\prod_{j<i} F_{j,k}^{-\mathsf{C}_{i,j}}\prod_{j>i} F_{j,k+1}^{-\mathsf{C}_{i,j}}
\]
Plugging the lower bound into the formula for $2a_i(F)$ gives
\[
\begin{aligned}
2a_i (F)
&>
\frac{1}{p} \log_2\left( \prod_{k=1}^p \left(
\prod_{j<i} F_{j,k}^{-\mathsf{C}_{i,j}}\prod_{j>i} F_{j,k+1}^{-\mathsf{C}_{i,j}}
\right)\right) 
=
\frac{1}{p} \log_2\left( \prod_{k=1}^p \left(
\prod_{j\neq i} F_{j,k}^{-\mathsf{C}_{i,j}}
\right)\right) \\
&=
\sum_{j \neq i}-\mathsf{C}_{i,j}
\left(\frac{1}{p} \log_2\left( 
\prod_{k=1}^p F_{j,k}
\right)\right)
=
-\sum_{j \neq i}\mathsf{C}_{i,j}
a_j(F)
\end{aligned}
\]
Plugging the upper bound into the formula for $2a_i(F)$ gives
\[
\begin{aligned}
2a_i (F)
&\leq
\frac{1}{p} \log_2\left( \prod_{k=1}^p \left(
2\prod_{j<i} F_{j,k}^{-\mathsf{C}_{i,j}}\prod_{j>i} F_{j,k+1}^{-\mathsf{C}_{i,j}}
\right)\right) 
=
\frac{1}{p} \log_2\left( 2^p\prod_{k=1}^p \left(
\prod_{j\neq i} F_{j,k}^{-\mathsf{C}_{i,j}}
\right)\right) \\
&=
1+
\sum_{j \neq i}-\mathsf{C}_{i,j}
\left(\frac{1}{p} \log_2\left( 
\prod_{k=1}^p F_{j,k}
\right)\right)
=1-
\sum_{j \neq i}\mathsf{C}_{i,j}
a_j(F)
\end{aligned}
\]
Therefore, for each $i\in \Gamma_0$, we have bounds
\[
0
< 2a_i(F)+\sum_{j \neq i}\mathsf{C}_{i,j}
a_j(F)
\leq
1
\]
Since the $i$th entry of $\mathsf{C}\vec{\mathbf{a}}(F)$ is equal to $2a_i(F)+\sum_{j\neq i} \mathsf{C}_{i,j}a_{j}(F) $, this completes the proof.
\end{proof}

\begin{ex}\label{ex: E8a}
Consider $\Gamma=E_8$, with Cartan matrix $\mathsf{C}$ and its inverse below.
\[
\mathsf{C}
:=
\begin{bmatrix}
2 & 0 & -1 & 0 & 0 & 0 & 0 & 0 \\
0 & 2 & 0 & -1 & 0 & 0 & 0 & 0 \\
-1 & 0 & 2 & -1 & 0 & 0 & 0 & 0 \\
0 & -1 & -1 & 2 & -1 & 0 & 0 & 0 \\
0 & 0 & 0 & -1 & 2 & -1 & 0 & 0 \\
0 & 0 & 0 & 0 & -1 & 2 & -1 & 0 \\
0 & 0 & 0 & 0 & 0 & -1 & 2 & -1 \\
0 & 0 & 0 & 0 & 0 & 0 & -1 & 2 \\
\end{bmatrix}
\hspace{.5cm}
\mathsf{C}^{-1}
=
\begin{bmatrix}
4 & 5 & 7 & 10 & 8 & 6 & 4 & 2 \\
5 & 8 & 10 & 15 & 12 & 9 & 6 & 3 \\
7 & 10 & 14 & 20 & 16 & 12 & 8 & 4 \\
10 & 15 & 20 & 30 & 24 & 18 & 12 & 6 \\
8 & 12 & 16 & 24 & 20 & 15 & 10 & 5 \\
6 & 9 & 12 & 18 & 15 & 12 & 8 & 4 \\
4 & 6 & 8 & 12 & 10 & 8 & 6 & 3 \\
2 & 3 & 4 & 6 & 5 & 4 & 3 & 2
\end{bmatrix}
\]
Note that this implicitly orders the vertices of $\Gamma$. For the frieze $F$ in Figure \ref{fig: E8frieze}, 
\[
\vec{\mathbf{a}}(F)
=
\begin{bmatrix}
1.79248 \\
2.57480 \\
3.45644 \\
5.10777 \\
4.18304 \\ 
3.25390 \\ 
2.30720 \\
1.29248 \\
\end{bmatrix}
\hspace{1cm}
\mathsf{C}\vec{\mathbf{a}}(F)
=
\begin{bmatrix}
0.12852 \\
0.04184 \\
0.01263 \\
0.00125 \\
0.00442 \\
0.01755 \\
0.06803 \\
0.27776
\end{bmatrix}
\qedhere
\]
\end{ex}

\subsection{Bounds}


Since $\Gamma$ is Dynkin, its Cartan matrix $\mathsf{C}$ is invertible with positive integer entries (e.g. \cite[Table 2]{OV90}).
%
Let $b_i$ denote the sum of the $i$th row of $\mathsf{C}^{-1}$; that is,
\[ b_i := \sum_{j=1}^n \mathsf{C}_{i,j}^{-1} \]

\begin{prop}\label{prop}
For all $i$ and $n$,
$a_i(F) \leq b_i$ and $F_{i,k}\leq 2^{pb_i}$.
%
\end{prop}



\begin{proof}
Since the entries of $\mathsf{C}^{-1}$ are positive, increasing any coordinate of a vector $\vec{\mathsf{v}}$ increases each coordinate of the product $\mathsf{C}^{-1}\vec{\mathsf{v}}$. Therefore, each coordinate attains its maximum value on $\mathsf{C}^{-1}[0,1]^n$ at 
$ \mathsf{C}^{-1} (1,1,..,1)$.
The $i$th coordinate of this point is $b_i$, so $a_i(F)\leq b_i$.
Raising both sides to the $p$, $\prod_{k=1}^pF_{i,k}\leq 2^{pb_i}$. Since each $F_{i,k}\geq1$, this implies that $F_{i,k}\leq 2^{pb_i}$.
%
\end{proof}

\begin{thm}
There are finitely many positive integral friezes of type $\Gamma$.
\end{thm}

\begin{proof}
Each $F_{i,k}$ must be an integer between $1$ and $2^{pb_i}$, and $F$ is determined by the values $F_{i,k}$ with $1\leq k\leq p$ and $1\leq i\leq n$. So, there are at most $2^{p^2\sum_i b_i}$-many friezes of type $\Gamma$.
\end{proof}



\section{Enumeration}

After finiteness, the natural question is: how many friezes are there of each type? 
\begin{itemize}
    \item As previously mentioned, friezes of type $A_n$ are counted by Catalan numbers.
    \item Types $B_n,C_n,D_n,$ and $G_2$ were counted by Fontaine and Plamondon in \cite{FP16b}.
    \item Fontaine and Plamondon also used a computer search to give conjectural counts in types $E_6,E_7,E_8,$ and $F_4$ \cite{FP16b}. The counts for $E_6$ and $F_4$ were confirmed in \cite{BFGST21}, but the counts of $4400$ and $26952$ in type $E_7$ and $E_8$ remain unconfirmed.
\end{itemize}
%
%
Unfortunately, while the techniques of this paper characterize finite spaces of possible friezes for the open types, these spaces are still impractically large to check with current computers.



\begin{ex}
Consider $\Gamma=E_8$. Summing the rows of $\mathsf{C}^{-1}$ (see Example \ref{ex: E8a}) gives
\[
(b_1,b_2,b_3,b_4,b_5,b_6,b_7,b_8)
=(
46,
68,
91,
135,
110,
84,
57,
29)
\]
Since every frieze of type $E_8$ has period $p=16$, such a frieze is determined by the entries $F_{i,k}$ for $1\leq k\leq 16$. Choosing a value for each such $F_{i,k}$ between $1$ and $2^{pb_i}$ gives
%
\[
\prod_{i=1}^8\prod_{k=1}^{16}  2^{16 b_i} = 2^{16^2 \prod_i b_i} =2^{158720} 
\]
possible friezes of type $E_8$. Oof.
%
%
%
There are a few optimizations available.
\begin{itemize}
    \item A frieze of type $E_8$ is (almost) determined by its values on Row $8$.\footnote{Choosing Row 8 determines Rows 7, 6, 5, and 4 by the mesh relations. The remaining rows require a choosing factorizations of some integers, but this adds a negligible number of cases to a computer search.}
    \item Instead of bounding the individual entries, bound the product of the entries in a row.
\end{itemize}
This reduces the problem to considering all $16$-tuples
$ F_{8,1},F_{8,2},F_{8,3},...,F_{8,16}$ in $\mathbb{Z}_{\geq1}$
with
\[\prod_{k=1}^{16} F_{8,k}\leq 2^{16 b_8} = 2^{16\cdot 29} = 2^{464}\]
Unfortunately, $2^{464}$ is still too many possibilities for a brute force search.


By \cite[Theorem C]{GM22},  Fontaine and Plamondon's conjectural counts would be confirmed by showing there are $0$ and $4$ friezes in types $E_7$ and $E_8$ (respectively) whose entries are all at least 2. 
For such friezes $F$, the upper bound on the $i$th entry of $\mathsf{C}\vec{\mathbf{a}}(F)$ can be reduced to 
$\log_2(1+2^{-d_i})$, where $d_i:=\sum_{j\neq i} -\mathsf{C}_{i,j}$. This translates into bounds
\[ a_i(F)\leq \sum_{j=1}^n \mathsf{C}_{i,j}^{-1}\log_2(1+2^{-d_j}) 
\text{ and }
\prod_{k=1}^pF_{i,k} \leq 
\prod_{j=1}^n(1+2^{-d_j})^{p\mathsf{C}_{i,j}^{-1}}
\]
For $\Gamma=E_8$ and $i=8$, this gives a bound
\[
\prod_{k=1}^{16}F_{8,k} \leq 
\left(\frac{151875}{16384}\right)^{16}
\approx 2^{51}
\]
This is approaching the realm of possibility, but further optimization is needed.
\end{ex}



\section*{Acknowledgements}

\begin{itemize}
\item Emily Gunawan, whose extensive collaborations on \cite{GM22} led to this idea.
\item Pierre-Guy Plamondon, who pushed the author to bound entries of a frieze.
\item Antoine de Saint Germain, for pushing the author to write this up.
\end{itemize}

\bibliographystyle{alpha}
\bibliography{myunibib}

\end{document}